\documentclass[oneside,10pt]{article}          
\usepackage[b5paper]{geometry}	    
\usepackage{amsfonts,amsmath,latexsym,amssymb
} 
\usepackage{amsthm}                
\usepackage{mathrsfs,upref}         
\usepackage{mathptmx}		    
\usepackage{jca}	            

\usepackage{hyperref}

\newtheorem{theorem}{Theorem}
\newtheorem{proposition}[theorem]{Proposition}

\newtheorem{lemma}[theorem]{Lemma}

\theoremstyle{definition}

\newcommand{\N}{\mathbb{N}}
\newcommand{\R}{\mathbb{R}}

\newcommand{\E}{\operatorname{\mathsf{E}}}

\newcommand{\ii}[1]{\operatorname{I}\{#1\}}
\newcommand{\iib}[1]{\operatorname{I}\Big\{#1\Big\}}

\newcommand{\vp}{\varepsilon}

\newcommand{\tf}{{\tilde{f}}}
\newcommand{\tg}{{\tilde{g}}}

\begin{document}

\title[A problem concerning Riemann sums]{A problem concerning Riemann sums}


\author{Iosif Pinelis}

\address{Iosif Pinelis, Department of Mathematical Sciences\\
Michigan Technological University\\
Hough\-ton, Michigan 49931, USA\\
\email{ipinelis@mtu.edu}}

\CorrespondingAuthor{Iosif Pinelis}


\date{
09.09.2018}                               




\begin{abstract}
        An open problem concerning Riemann sums, posed by O.\ Furdui, is considered. 
\end{abstract}

\maketitle



Let $f\colon[0,1]\to\R$ be a continuous function. For natural $n$, let 
\begin{equation}\label{eq:x,y}
	x_n:=\sum_{k=1}^{n-1}f\Big(\frac kn\Big)\quad\text{and}\quad y_n:=x_{n+1}-x_n. 
\end{equation}
One may note here that $x_n/n$ is a Riemann sum approximating the integral $\int_0^1 f(x)\,dx$. 

Part (a) of Problem~1.32 in the book \cite{furdui} is to find $\lim_{n\to\infty}y_n$ 
if the function $f$ is continuously differentiable. It is not hard to do a bit more: 

\begin{proposition}\label{prop:abs cont}
Whenever $f$ is absolutely continuous, one has 
\begin{equation}\label{eq:}
	\lim_{n\to\infty}y_n=\int_0^1 f(x)\,dx. 
\end{equation}

\end{proposition}


On the other hand, it is even easier to show this: 

\begin{proposition}\label{prop:uniq}
Whenever $\lim_{n\to\infty}y_n$ exists,  
equality \eqref{eq:} holds. 
\end{proposition}

Propositions \ref{prop:abs cont} and \ref{prop:uniq} will be proved at the end of this note. 

\bigskip

Part (b) of 
Problem~1.32 in \cite{furdui} is the following 
question, which has so far remained apparently unanswered: 
\begin{quote}
\centering
{\normalsize \emph{What is the limit \emph{[in \eqref{eq:}]} when $f$ is only continuous?}}
\end{quote}
By Proposition \ref{prop:uniq}, this limit, if it exists, may only be $\int_0^1 f(x)\,dx$. However, we have 

\begin{theorem}\label{th:}
There are continuous functions $f\colon[0,1]\to\R$\ \, for which $\lim_{n\to\infty}y_n$ does not exist. 
\end{theorem}

This is a pure existence theorem, and its proof, given below, is non-constructive. 
%
So, the problem of explicitly constructing a continuous function for which $\lim_{n\to\infty}y_n$ does not exist remains open. 

\begin{proof}[Proof of Theorem~\ref{th:}]
It may 
come as a surprise that this proof uses a probabilistic method.  
Let $f$ be the random function $W$ that is a standard Wiener process (Brownian motion) over the interval $[0,1]$. Then the 
probability that $f$ is continuous everywhere on $[0,1]$ is $1$; see e.g.\ \cite{morters-peres}. In fact, without loss of generality we may assume that all realizations of the random function $f=W$ are everywhere continuous. 

Informally, the idea of this proof of Theorem~\ref{th:} is that, while all realizations of $W$ are everywhere continuous, they are rather non-smooth, not only in the sense of being nowhere differentiable, but also not being H\"older continuous with any exponent $\ge1/2$, in view of the (local) law of the iterated logarithm \cite{morters-peres}. 

Now back to the formal proof: actually, Theorem~\ref{th:} follows 
immediately from 

\begin{lemma}\label{lem:} For $f=W$, the distribution of the random variable $y_{4s}-y_{2s}$ converges to the centered normal distribution with variance $1/4$. 
\end{lemma}
\noindent\rule{0pt}{0pt}\big(The convergence here and in the rest of the proof of Theorem~\ref{th:} is as $\N\ni s\to\infty$.\big) 
Indeed, if Theorem~\ref{th:} were false, we would have $y_{4s}-y_{2s}\to0$ almost surely and hence in distribution, which would contradict Lemma~\ref{lem:}. 
So, to complete the proof of Theorem~\ref{th:}, it remains to prove the lemma. 

\begin{proof}[Proof of Lemma~\ref{lem:}] Since $y_{4s}-y_{2s}$ is a centered normal random variable, it suffices to show that 
\begin{equation}\label{eq:to1/4}
	\E(y_{4s}-y_{2s})^2\overset{\text{(?)}}
	\longrightarrow
	1/4. 
\end{equation}
The proof of \eqref{eq:to1/4} consists in direct calculations, which are somewhat involved, though, as we have to deal carefully enough with the discreteness in the definition of $x_n$. In carrying out this task, the choice of indices, $4s$ and $2s$, in the statement of Lemma~\ref{lem:} turns out to be sufficiently convenient.  
 
The just mentioned calculations are based on the formula 
\begin{equation*}
\E W(u)W(v)=u\wedge v	
\end{equation*}
for all $u,v$ in $[0,1]$. By \eqref{eq:x,y}, for $f=W$, 
\begin{equation}\label{eq:Ex_n^2}
	\E x_n^2=\sum_{j,k=1}^{n-1}\E W\Big(\frac jn\Big)W\Big(\frac kn\Big)
	=
	\sum_{j,k=1}^{n-1}\Big(\frac jn\wedge \frac kn\Big)=\frac{2 n^2-3 n+1}6. 
\end{equation}
Somewhat similarly,
\begin{equation}\label{eq:Ex_n x_{n+1}}
\begin{aligned}
		\E x_n x_{n+1}
	&=\sum_{j=1}^{n-1}\sum_{k=1}^n\Big(\frac jn\wedge \frac k{n+1}\Big) \\ 
	&=\sum_{j=1}^{n-1}\sum_{k=j+1}^n\frac jn 
	+\sum_{j=1}^{n-1}\sum_{k=1}^j\frac k{n+1}  
	=\frac{2 n^2-n-1}6. 
\end{aligned}	
\end{equation}
It follows from \eqref{eq:x,y}, \eqref{eq:Ex_n^2}, and \eqref{eq:Ex_n x_{n+1}} that 
\begin{equation}\label{eq:1/2}
	\E y_n^2=\E x_{n+1}^2+\E x_n^2-2\E x_n x_{n+1}=1/2. 
\end{equation}

Now take any natural $s$. Similarly to \eqref{eq:Ex_n x_{n+1}}, we have 
\begin{equation*}
\begin{aligned}
	\E x_{4s} x_{2s}
	&=\sum_{j=1}^{4s-1}\sum_{k=1}^{2s-1}\Big(\frac j{4s}\wedge \frac k{2s}\Big) \\ 
	&=\sum_{k=1}^{2s-1}\sum_{j=1}^{2k}\frac j{4s}
	+\sum_{k=1}^{2s-1}\sum_{j=2k+1}^{4s-1}\frac k{2s}  
	=\frac{32 s^2-18 s+1}{12}, 	\\ 	
	\E x_{4s+1} x_{2s+1}
	&=\sum_{j=1}^{4s}\sum_{k=1}^{2s}\Big(\frac j{4s+1}\wedge \frac k{2s+1}\Big) \\ 
	&=\sum_{k=1}^{2s}\sum_{j=1}^{2k-1}\frac j{4s+1}
	+\sum_{k=1}^{2s}\sum_{j=2k}^{4s-1}\frac k{2s+1}  
	=\frac{32 s^3+14s^2}{12 s+3}, \\ 	
\end{aligned}	
\end{equation*}
\begin{equation*}
\begin{aligned}
	\E x_{4s+1} x_{2s}
	&=\sum_{j=1}^{4s}\sum_{k=1}^{2s-1}\Big(\frac j{4s+1}\wedge \frac k{2s}\Big) \\ 
	&=\sum_{k=1}^{2s-1}\sum_{j=1}^{2k}\frac j{4s+1}
	+\sum_{k=1}^{2s-1}\sum_{j=2k+1}^{4s}\frac k{2s}  
	=\frac{32 s^3-2 s^2-5 s-1}{12 s+3}, \\ 		
	%
	\E x_{4s} x_{2s+1}
	&=\sum_{j=1}^{4s-1}\sum_{k=1}^{2s}\Big(\frac j{4s}\wedge \frac k{2s+1}\Big) \\ 
	&=\sum_{k=1}^{s}\sum_{j=1}^{2k-1}\frac j{4s}
	+\sum_{k=1}^{s}\sum_{j=2k}^{4s-1}\frac k{2s+1} \\  
	&+\sum_{k=s+1}^{2s}\sum_{j=1}^{2k-2}\frac j{4s}
	+\sum_{k=s+1}^{2s}\sum_{j=2k-1}^{4s-1}\frac k{2s+1} 
	=\frac{64 s^3+28 s^2-7 s-1}{24 s+12}.  		
\end{aligned}	
\end{equation*}
So,
\begin{align*}
	\E y_{4s}y_{2s}&=\E x_{4s}x_{2s}+\E x_{4s+1}x_{2s+1}-\E x_{4s+1}x_{2s}-\E x_{4s}x_{2s+1} \\ 
	&=\frac{12 s^2+9 s+2}{32 s^2+24 s+4}\longrightarrow\frac{12}{32}.    
\end{align*}
Thus, in view of \eqref{eq:1/2}, 
\begin{equation*}
	\E(y_{4s}-y_{2s})^2=\E y_{4s}^2+\E y_{2s}^2-2\E y_{4s}y_{2s}\longrightarrow\frac12+\frac12-2\times\frac{12}{32}=\frac14, 
\end{equation*}
so that \eqref{eq:to1/4} is verified, which completes the proof of Lemma~\ref{lem:}. 
\end{proof}
The proof of Theorem~\ref{th:} is now complete as well. 
\end{proof}

To conclude this note, it remains to prove Propositions \ref{prop:abs cont} and \ref{prop:uniq}. 

\begin{proof}[Proof of Proposition \ref{prop:abs cont}]
Since $f$ is absolutely continuous, there is a function $g\in L^1[0,1]$ 
such that 
\begin{equation}\label{eq:f}
f(x)=f(0)+\int_0^x g(u)\,du=f(0)+\int_0^1 g(u)\ii{u<x}\,du	
\end{equation}
for all $x\in[0,1]$, where $\ii\cdot$ denotes the indicator. So, by \eqref{eq:x,y}, 
\begin{equation*}
	x_n=(n-1)f(0)+\int_0^1 g(u)\sum_{k=1}^{n-1}\iib{u<\frac kn}\,du
\end{equation*}
and hence 
\begin{equation}\label{eq:y_n=}
	y_n=f(0)+I_n(g)-J_n(g),
\end{equation}
where 
\begin{equation*}
	I_n(g):=\int_0^1 g(u)\iib{u<\frac n{n+1}}\,du,\quad J_n(g):=\int_0^1 g(u)h_n(u)\,du,
\end{equation*}
and 
\begin{equation}\label{eq:h_n}
	h_n(u):=\sum_{k=1}^{n-1}\iib{\frac k{n+1}\le u<\frac kn}.
\end{equation}
Clearly,
\begin{equation}\label{eq:I_n to}
	I_n(g)=f\Big(\frac n{n+1}\Big)-f(0)\longrightarrow f(1)-f(0). 
\end{equation}
Here and in the rest of this proof, the convergence is as $n\to\infty$. 

To deal with $J_n(g)$, take any real $\vp>0$. Since $g\in L^1[0,1]$, by \cite[Corollary~4.2.2]{bogachev}, $\int_0^1 |g(u)-\tg(u)|\,du\le\vp$ for some continuous function $\tg\colon[0,1]\to\R$. 
Note also that $0\le h_n\le1$, since $[\frac k{n+1},\frac kn)\subset[\frac{k-1}n,\frac kn)$ for $k=1,\dots,n-1$. 
So,    
\begin{equation}\label{eq:<ep}
	|J_n(g)-J_n(\tg)|\le\int_0^1 |g(u)-\tg(u)|\,du\le\vp. 
\end{equation}
Introduce now the function $\tg_n$ by the formula 
$$\tg_n(u):=\sum_{k=1}^{n-1}\tg\Big(\frac kn\Big)\iib{\frac k{n+1}\le u<\frac kn}$$ 
for $u\in[0,1]$. 
Since the function $\tg$ is continuous, it is uniformly continuous on $[0,1]$, so that, in view of \eqref{eq:h_n}, $\|\tg h_n-\tg_n h_n\|_\infty=\|\tg h_n-\tg_n\|_\infty\to0$ and hence 
\begin{equation}\label{eq:J-J}
	|J_n(\tg)-J_n(\tg_n)|\longrightarrow0. 
\end{equation}
On the other hand, using the continuity of $\tg$ and integration by parts, we have 
\begin{equation}\label{eq:J_n to}
	J_n(\tg_n)=\sum_{k=1}^{n-1}\tg\Big(\frac kn\Big)\,\frac kn\,\frac 1n\;\frac n{n+1}\longrightarrow
	\int_0^1\tg(u) u\,du
	=\tf(1)-\int_0^1\tf(u)\,du, 
\end{equation}
where $\tf(x):=f(0)+\int_0^x\tg(u)\,du$ for $x\in[0,1]$. 
By \eqref{eq:f} and the second inequality in \eqref{eq:<ep}, we have $|\tf-f|\le\vp$ and hence $|\int_0^1\tf(u)\,du-\int_0^1 f(u)\,du|\le\vp$. Collecting now \eqref{eq:y_n=}, \eqref{eq:I_n to}, \eqref{eq:<ep}, \eqref{eq:J-J}, and \eqref{eq:J_n to}, we see that 
\begin{equation*}
	\limsup_{n\to\infty}\Big|y_n-\int_0^1 f(u)\,du\Big|\le3\vp,
\end{equation*}
for any real $\vp>0$, which completes the proof of Proposition \ref{prop:abs cont}. 
\end{proof}

\begin{proof}[Proof of Proposition \ref{prop:uniq}]
The Stolz--Ces\`aro theorem (\cite[pages~173--175]{stolz} and \break 
\cite[page~54]{cesaro}) states the following: if $(a_n)$ and $(b_n)$ are sequences of real numbers such that $b_n$ is strictly increasing to $\infty$ and $\frac{a_{n+1}-a_n}{b_{n+1}-b_n}\longrightarrow\ell\in\R$, then $\frac{a_n}{b_n}\longrightarrow\ell$. 
Now Proposition \ref{prop:uniq} follows immediately by applying the Stolz--Ces\`aro theorem with $a_n=x_n$ and $b_n=n$, since $\frac{x_n}{n}\longrightarrow\int_0^1 f(x)\,dx$. 
\end{proof}


\bibliographystyle{abbrv}

\bibliography{P:/pCloudSync/mtu_pCloud_02-02-17/bib_files/citations12.13.12}
%
%
%
%
%
%
%
%

\end{document}